\documentclass[11pt,leqno]{amsart}
\usepackage{verbatim,paralist}
\usepackage[breaklinks,pdfstartview=FitH]{hyperref}

\usepackage{amssymb,mathrsfs,verbatim}

\theoremstyle{plain}
  \newtheorem*{proposition}{Proposition}

\begin{document}

\title{A note on extensions of approximate ultrametrics}

\author{Manor Mendel}
\thanks{
Written while the author was a member of the Institute for Advanced Study
at Princeton NJ. Partially supported by ISF grant 93/11, BSF grant 2010021, NSF grant for "Center for intractability" and NSF CDI grant on pseudorandomness.} 
\address{Mathematics and Computer Science, Open University of Israel} 
\email{mendelma@gmail.com}

\maketitle

An ultrametric $\rho:X\times X\to [0,\infty)$ is a metric satisfying a strong form of the triangle
inequality:
\[ \rho(x,z)\le \max\{\rho(x,y),\rho(y,z)\}, \quad \forall x,y,z\in X.\]
It  $D$-approximates a metric $d:X\times X\to [0,\infty)$ if there exists a scaling factor $c>0$
such that
\[ d(x,y)\le c\cdot \rho(x,y) \le D\cdot d(x,y), \quad \forall x,y\in X.\]

Let $(X,d)$ be a metric space, $S\subset X$ and suppose that $S$ has an ultrametric 
$\rho:S\times S\to [0,\infty)$ that $D$ approximates $d$ on $S$. 
It was observed in~\cite[Lemma~4.1]{MN07}
that $\rho$ can be extended to an ultrametric $\bar\rho:X\times X\to[0,\infty)$ which $6D$ approximates
$d$ on pairs from $S\times X$. In this note we optimize the (simple) argument from~\cite{MN07}, 
obtaining an extension  $\bar\rho:X\times X\to[0,\infty)$ that $2D+1$ approximates
$d$ on pairs from $S\times X$. This bound is optimal in the worst case when $D\in \mathbb N$.
While the following proposition is true for an arbitrary metric space $X$ and an arbitrary subset $S\subset X$, we will assume for simplicity that $X$ is locally compact and $S$ is closed.

\begin{proposition}
Let $(X,d)$ be a locally compact metric space. Given a closed subset $S\subset X$ and an ultrametric 
$\rho:S\times S\to [0\infty)$ on $S$ satisfying
\[ d(x,y)\le \rho(x,y) \le D\cdot d(x,y), \quad \text{for every } x,y\in S,\]
there exists an ultrametric $\bar \rho:X\times X\to[0,\infty)$ extending $\rho$, such that
\begin{align}
\bar \rho(x,y) &= \rho(x,y), & \forall x,y\in S;\\
\frac{2D}{2(D+1)} \cdot d(x,y) & \le \bar\rho(x,y)   & \forall x,y\in X;\\
 \frac{2D}{2D+1} \cdot d(x,y) & \le \bar\rho(x,y) \le 2D\cdot d(x,y) & \forall x\in S,\, y\in X.
\end{align}
In particular, $\bar \rho$ is $2D+1$ approximation of $d$ on pairs from $S\times X$.
\end{proposition}
\begin{proof}
For $y\in X$, let $N(y)\in S$ be its nearest neighbor in $S$, i.e., $d(y,N(y))=\min_{x\in S} d(x,y)$ (breaking
ties arbitrarily).
Define an ultrametric $\bar \rho $ on $X$ as follows:
\[ \bar\rho(x,y)= \max\{2D\cdot d(x,N(x)), 2D\cdot d(y,N(y)), \rho(N(x),N(y))\}.\]
Clearly, the restriction of $\bar \rho $ to $S$ is equal $\rho$. $\bar \rho$ is an ultrameric since,
\begin{align*} 
\bar\rho(x,z) &=
\max\{2D\cdot d(x,N(x)), 2D\cdot d(z,N(z), \rho(N(x),N(z))\}
\\ & \le 
\max \bigl \{2D\cdot d(x,N(x)), 2D\cdot d(z,N(z), \\ & \qquad \qquad 2D\cdot d(y,N(y)), \rho(N(x),N(y)), \rho(N(y),N(z)) \bigr\}
\\ &= \max \bigl \{
\max\{2D\cdot d(x,N(x)),  2D\cdot d(y,N(y)), \rho(N(x),N(y))\} 
\; ,\\ &
\qquad \qquad \max\{ 2D\cdot d(z,N(z), 2D\cdot d(y,N(y)),  \rho(N(y),N(z))\}
\bigr\}
\\ &= \max\{\bar \rho(x,y), \bar \rho(y,z)\}.
\end{align*}

We next show that $\bar\rho(x,y)\ge \frac{2D}{2(D+1)} \cdot d(x,y) $ for every $x,y\in X$.
\begin{multline*}
\bar\rho(x,y)= \max\{2D\cdot d(x,N(x)), 2D\cdot d(y,N(y), \rho(N(x),N(y))\}
\\ \ge \tfrac{2D}{2(D+1)} d(x,N(x)) + \tfrac{2D}{2(D+1)} d(y,N(y))+ (1-\tfrac{2}{2(D+1)}) d(N(x),N(y))
\\ \ge \tfrac{2D}{2(D+1)} \cdot d(x,y).
\end{multline*}

Similarly for every $x\in S$ and $y\in X$,
\begin{multline*}
\bar\rho(x,y)= \max\{2D\cdot d(x,N(x)), 2D\cdot d(y,N(y), \rho(N(x),N(y))\}
\\ \ge \tfrac{2D}{2D+1} d(y,N(y))+ (1-\tfrac{1}{2D+1}) d(N(x),N(y))
\\ \ge \tfrac{2D}{2D+1} \cdot d(x,y).
\end{multline*}

Lastly, assume $x\in S$, and $y\in X$. By definition,
$d(y,N(y))\le d(y,x)$, and so by the triangle inequality,
$d(x,N(y))\le d(x,y)+d(y,N(y))\le 2 d(x,y)$. Hence,
\begin{multline*}
\bar \rho(x,y)= \max\{2D\cdot d(y,N(y)), \rho(x,N(y)) \}
\\ \le \max\{2D\cdot d(y,N(y)), D\cdot d(x,N(y))\} 
\\ \le \max\{2D\cdot d(y,N(y)), 2D\cdot d(x,y)\}= 2D\cdot d(x,y). \qedhere
\end{multline*}
\end{proof}

The $2D+1$ bound on the approximation of pairs from $S\times X$ 
above is tight for every $D\in \mathbb N$ 
as the following example shows:
\[ X=\{0,1,\ldots, 2D+1\}, \quad S=\{1,3,\ldots, 2D+1\} .\]
Since $S$ is a set of $D+1$ equally spaced points on the line, it is $D$ approximate ultrametric.
On the other hand, for any ultrametric $\bar \rho$ on $X$ which dominates the line distances in $S\times X$ we have
\[ \max_{i\in\{0,\ldots,2D\}} \bar \rho(i,i+1)\ge \bar \rho(0,2D+1)\ge 2D+1.\]

\end{document}